\pdfoutput=1
\documentclass[fontsize=12pt,paper=a4]{scrartcl}

\usepackage[T1]{fontenc}
\usepackage[utf8]{inputenc}

\usepackage{lmodern}
\usepackage{microtype}

\makeatletter
\renewcommand*\author[1]{%
  \stepcounter{author}%
  \ifnum\c@author=1
    \gdef\@author{#1}%
  \else
    \xdef\@author{\unexpanded\expandafter{\@author\and#1}}%
  \fi
  \csgdef{author@\the\c@author}{#1}}
\newcommand*\email[1]{%
  \csgdef{email@\the\c@author}{#1}}
\newcommand*\address[1]{%
  \csgdef{address@\the\c@author}{#1}}
\AtEndDocument{%
  \xdef\author@count{\the\c@author}%
  \c@author=1
  \print@authors}
\newcommand*\print@authors{%
  \ifnum\c@author>\author@count
  \else
    \print@author{\the\c@author}%
    \advance\c@author by 1
    \expandafter\print@authors
  \fi}
\newcommand*\print@author[1]{%
  \par\medskip
  \noindent
  \begin{tabular}{@{}l@{}}%
    \csuse{author@#1}\\
    \csuse{address@#1}\\
    \textit{e-mail:}
    \texttt{\csuse{email@#1}}
  \end{tabular}
}

\newcommand{\fundingfootnote}[1]{%
        \begin{NoHyper}%
           \renewcommand{\thefootnote}{}
           \footnote{#1}%
           \addtocounter{footnote}{-1}%
        \end{NoHyper}%
}

\renewcommand*{\title}[2][]{%
  \ifstrempty{#1}{
  \gdef\shorttitle{#2}\gdef\@title{#2}
  }{\gdef\shorttitle{#1}\gdef\@title{#2}%
  }
}
\makeatother

\RequirePackage{xpatch}

\RequirePackage{scrlayer-scrpage}

\ohead{\pagemark}
\makeatletter 
\cehead[]{\let\thanks\@gobble\@author}
\cohead[]{\shorttitle}
\makeatother
\ofoot[]{}

\makeatletter
\xpatchcmd{\swappedhead}%
         {~}%
         {.~}
         {}%
         {\typeout{failed to patch swappedhead}}
\makeatother

\KOMAoptions{twoside=semi,headings=small, numbers=enddot}

\setkomafont{disposition}{\normalcolor\bfseries}
\addtokomafont{title}{} 

\setkomafont{author}{\large}
\setkomafont{date}{\normalfont}

\newkomafont{abstract}{\normalfont\footnotesize}
\newkomafont{abstracttitle}{%
   \usekomafont{disposition}\usesizeofkomafont{abstract}%
}

\renewenvironment{abstract}%
  {\usekomafont{abstract}
   \begin{description}[%
        leftmargin=3em,
        labelindent=3em,
        rightmargin=3em,
        labelwidth=0pt,
        parsep=0pt,
        listparindent=\parindent, 
        style=sameline,
        font=\normalfont\usekomafont{abstracttitle}]
   \item[\abstractname.]}%
  {\end{description}}

\newcommand*{\keywordsname}{Keywords}
\newcommand*{\mscname}[1]{Mathematics Subject Classification (#1)}
  
\newcommand{\keywords}[1]{\item[\keywordsname.] #1.}  
\newcommand{\subjclass}[2][2010]{%
\item[\mscname{#1}.] #2.}

\KOMAoptions{DIV=12}

\usepackage[style=numams, 
            sorting=nyt, 
            isbn=false,  
            backref=true, 
            maxbibnames=5]{biblatex}
\usepackage{amsmath,amssymb, amsthm}
\usepackage{mathtools}
\usepackage{enumitem}   

\usepackage[%
         pdfa,
         colorlinks,linkcolor=blue, citecolor=blue
         ]{hyperref}

\usepackage[capitalize]{cleveref}

\swapnumbers
\newtheorem{thm}{Theorem}[section]
\newtheorem{cor}[thm]{Corollary}
\newtheorem{lemma}[thm]{Lemma}

\theoremstyle{definition}
\newtheorem{defi}[thm]{Definition}
\newtheorem{remark}[thm]{Remark}

\newlist{enumthm}{enumerate}{1} 
\setlist[enumthm]{label= \textup{(\roman*)}}

\renewcommand{\geq}{\geqslant}
\renewcommand{\leq}{\leqslant}

\renewcommand{\phi}{\varphi}
\newcommand{\eps}{\varepsilon}

\newcommand{\defemph}[1]{\textbf{#1}} 

\newcommand{\ints}{\mathbb{Z}}
\newcommand{\reals}{\mathbb{R}}

\newcommand{\iso}{\cong}    

\DeclarePairedDelimiterX{\skp}[2]{\langle}{\rangle}{#1,#2} 
\newcommand{\eskp}{\skp{\:}{\,}}    

\DeclarePairedDelimiter{\abs}{\lvert}{\rvert}
\DeclarePairedDelimiter{\norm}{\lVert}{\rVert}
\DeclarePairedDelimiter{\erz}{\langle}{\rangle}

\DeclarePairedDelimiterX{\menge}[2]{\{}{\}}{
      \, #1 : %
      \allowbreak\nonscript\, \mathopen{} #2 \,}

\DeclareMathOperator{\Ker}{Ker}        
\DeclareMathOperator{\LinH}{span}
\DeclareMathOperator{\id}{Id}
\DeclareMathOperator{\latt}{\mathbb{L}}
\DeclareMathOperator{\mindis}{MinDist}
\DeclareMathOperator{\Mvecs}{MinVecs}

\newcommand*{\mytitle}{%
   Lattices of %
   finite abelian groups}
\newcommand*{\myauthor}{Frieder Ladisch}
\newcommand*{\mykeywords}{Eutactic lattice,
  minimal vectors, finite abelian group}

\hypersetup{pdfauthor={\myauthor},
            pdftitle={\mytitle},
            pdfkeywords={\mykeywords}
            }

\title{\mytitle}                           
\author{\myauthor%
}
\address{Universität Rostock \\
         Institut für Mathematik \\
         18051 Rostock (Germany)}
\email{frieder.ladisch@uni-rostock.de}

\begin{document}
\date{}
\maketitle

\begin{abstract}
\fundingfootnote{%
        Author supported by Deutsche Forschungsgemeinschaft (DFG),
        Project SCHU1503/7-1.}%
  We study certain lattices constructed from finite abelian groups.
  We show that such a lattice is eutactic,
  thereby confirming a conjecture by
  Böttcher, Eisenbarth, Fukshansky, Garcia, Maharaj.
  Our methods also yield simpler proofs of two known results:
  First, such a lattice is strongly eutactic if and only if
  the abelian group has odd order or is elementary abelian.
  Second, such a lattice has a basis of minimal vectors,
  except for the cyclic group of order~$4$.
\end{abstract}

\section{Introduction}
For each finite abelian group~$A$, one can construct a certain lattice
$\latt(A)$ of rank $\abs{A}-1$.
(We recall the construction of $\latt(A)$ in \cref{sec:lattA} below.)
For cyclic~$A$, the lattice $\latt(A)$
probably appeared first in a paper
by E.~S.~Barnes~\cite{Barnes57} and is therefore often called
the \emph{Barnes lattice}.
The construction of $\latt(A)$ also generalizes a family 
of lattices coming from elliptic curves over finite
fields~\cite{FukshanskyMaharaj14}.
For general abelian groups~$A$,
the lattice~$\latt(A)$ has been studied by
A.~Böttcher, S.~Eisenbarth, L.~Fukshansky, S.~R.~Garcia,
H.~Maharaj \cite{BoettcherEFGM19,BoettcherFGM15},
and by R. Bacher~\cite{Bacher15}, among others.

One of the results of Böttcher et al.~\cite{BoettcherEFGM19} is
that the lattice $\latt(A)$ is strongly eutactic if and only if
$A$ has odd order or is an elementary abelian $2$-group.
They conjectured that $\latt(A)$ is always eutactic
(except when $A\iso C_4$, the cyclic group of order~$4$).
The main result of this paper,
\cref{t:la_eutactic}, is that this conjecture is true.

Eutactic lattices are interesting due to a connection to 
sphere packings.
By a classical result of Voronoi~\cite[Theorem~4.6.3]{Martinet03PLES},
a lattice is extreme, 
that is, a local maximum of the packing density,
if and only if it is eutactic and perfect.
Moreover, A.~Schürmann~\cite{Schuermann10} has shown that perfect,
strongly eutactic lattices are periodic extreme. 
R. Bacher~\cite{Bacher15} has shown that $\latt(A)$ is perfect,
except for some groups~$A$ of order at most~$8$.
Thus $\latt(A)$ is extreme, except for some $A$ with $\abs{A}\leq 8$.

One key ingredient in our proof is the observation
that the lattice $\latt(A)$
is in fact the square of the augmentation ideal of the group ring
(see \cref{l:latt_augsquare}).
From this observation, 
we derive in \cref{sec:minL} 
a very useful parametrization of the set of  minimal vectors
of $\latt(A)$ 
(the set of nonzero vectors in $\latt(A)$ of minimal length).
Böttcher, Fukshansky, Garcia, 
and Maharaj~\cite{BoettcherFGM15}
showed that the lattice $\latt(A)$ 
is generated by minimal vectors
(except when $A\iso C_4$),
and has in fact a basis consisting of minimal vectors.
In general, a lattice generated by minimal vectors 
need not have a basis of 
minimal vectors~\cite{ConwaySloane95}.
Such lattices exist in dimension $\geq 10$,
but not in lower dimensions~\cite{MartinetSchuermann12}.
The lattices considered in this paper, however, 
do not exhibit this phenomenon.
In \cref{sec:basmin}, we use our methods to give a short and 
very elementary proof
that $\latt(A)$ ($A\not\iso C_4$)
has a basis of minimal vectors.
The proof is by direct computation in the group ring,
and we indicate how to find different bases of minimal vectors.
\Cref{sec:basmin} is not needed for the final \cref{sec:eutaxy},
where eutaxy is treated.

\section{The lattice of a finite abelian group}
\label{sec:lattA}

Let $A$ be a finite abelian group, written
\emph{multiplicatively}.
The \defemph{group ring} $\ints A$ of $A$ is by definition the set
of formal sums
\[ \ints A := 
    \menge[\bigg]{\sum_{a\in A} r_a a }{ r_a\in \ints } \,.
\]
Multiplication in $\ints A$ is induced by the multiplication in the 
group~$A$, extended distributively.
The \defemph{augmentation ideal} $\Delta A$ of the group ring
is the kernel of the augmentation map 
(coefficient sum) $\ints A \to \ints$, that is,
\[ \Delta A := 
   \menge[\bigg]{ \sum_{a\in A} r_a a 
      }{ \sum_{a\in A} r_a = 0, \, r_a\in \ints } \,.
\]
We consider the exponential type homomorphism
\[\psi_A \colon \ints A \to A
  , \quad
  \psi_A \big( \sum_{a\in A}k_a a \big) = \prod_{a\in A} a^{k_a}.
\] 

\begin{defi}
The \defemph{lattice} $\latt(A)$ of the finite abelian group $A$
is defined as 
\[ \latt(A) := \Ker(\psi_A) \cap \Delta A.
\]
\end{defi}

This is the lattice that has been studied before by
A.~Böttcher, S.~Eisenbarth, L.~Fukshansky, S.~R.~Garcia,
H. Maharaj~\cite{BoettcherEFGM19,BoettcherFGM15},
and R. Bacher~\cite{Bacher15}, 
with the difference that additive notation for the abelian group~$A$
is used in the cited papers.
In the present paper, 
we employ multiplicative notation for $A$ since then the notation for the
group ring is more convenient.

\begin{lemma}\label{l:latt_augsquare}
  $\latt(A) =  (\Delta A)^2$,
  the square of the augmentation ideal.
\end{lemma}

By definition, the square of an ideal $I$ in a ring is the
additive subgroup generated by all products $ab$ with
$a$, $b\in I$.

\begin{proof}[Proof of \cref{l:latt_augsquare}]
The elements of the form
$a-1$, $1\neq a\in A$, form a $\ints$-basis of $\Delta A$.
Thus $(\Delta A)^2$ as $\ints$-module is spanned by elements
of the form $(a-1)(b-1)$ with $a$, $b\in A$.
As these elements are in $\Ker \psi_A$,
it follows that $(\Delta A)^2 \subseteq \latt(A)$.

For the converse inclusion, we need 
the following congruence which seems to be well known
\cite[Lemma~2.1]{CliffSehgalWeiss81}, \cite[p.~7]{Hoechsmann92}:
\[ d  \equiv \psi_A(d)-1 \mod (\Delta A)^2
   \quad \text{for all $d \in \Delta A$}. 
\]
This congruence follows from repeatedly applying the
identities
\begin{align*}
   (a-1) + (b-1) &= (ab-1) - (a-1)(b-1) \quad\text{and}
   \\
   (a-1) - (b-1) &= (ab^{-1}-1) - (a-b)(b^{-1}-1).
\end{align*}
The congruence yields $\psi_A(d) = 1 \implies d\in (\Delta A)^2$
for $d\in \Delta A$, and thus the lemma.            
\end{proof}

\section{Elements of minimal length}
\label{sec:minL}

We consider $\ints A$ as a subset of $\reals A$, 
and $\reals A$ as a Euclidean space with respect to the usual scalar
product
\[ \left\langle
    \sum_{a\in A} x_a a, \sum_{a\in A} y_a a 
    \right\rangle
   = \sum_{a\in A} x_a y_a.
\]
As usual,
$\norm{x} := \sqrt{ \langle x, x \rangle }$,
this is called the length or Euclidean norm of $x\in \reals A$.

The \defemph{minimum distance} of a lattice $\Lambda$
is defined as
\[ \mindis(\Lambda):= \min \menge{\norm{x}}{
                              x\in \Lambda,\, x\neq 0
                           }.
\]

As a $\ints$-module, 
$(\Delta A)^2$ is spanned by elements of the form
\[  (a-1)(b-1) = ab -a -b + 1, \quad a, b \in A \,.
\]
We compute their length:
\begin{lemma}\label{l:normprodd2a}
  Let $a\neq 1 \neq b \in A$.
  Then
  \[ \norm{(a-1)(b-1)}^2 = 
      \begin{cases}
        8, \quad & \text{if } a=b=a^{-1}, \\
        6,       & \text{if $a=b$ or $a=b^{-1}$, but $a^2\neq1$,}
        \\
        4  & \text{else.} 
      \end{cases}
  \]
\end{lemma}
\begin{proof}
  Since $a\neq 1 \neq b$, we also have $a\neq ab\neq b$.
  Thus the only equalities that can occur among
  $1$, $a$, $b$, $ab$ are
  $a=b$ and/or $ab=1$.
\end{proof}

From the equality $\latt(A) = \Ker \psi_A \cap \Delta A$, 
it is immediate that every nonzero element 
in $\latt(A)$ has length at least 
$\sqrt{1^2+1^2+1^2+1^2}=2$, and this length is attained 
exactly at the elements of the form
\[ x + y - r-s \in \ints A
   \quad \text{such that }
   x,\, y,\, r,\, s \in A
   \text{ are distinct and } xy = rs.
\]
We want to parametrize the vectors of length~$2$ in a different way:
\begin{lemma}\label{l:minLA_param}
  Write
  \[ \Omega(A):= 
     \menge{(a,b)\in A\times A}{a\neq 1 \neq b \neq a^{\pm 1}}.  
  \]
  Define $f\colon \Omega(A)\times A \to \latt(A)$ by
  \[ f(a,b,g):= (a-1)(b-1)g,
     \quad (a,b)\in \Omega(A), g\in A.
  \]
  Then $f$ defines a $4$-to-$1$ map from $\Omega(A)\times A$
  onto the set
  \[ \menge{v\in \latt(A)}{\norm{v}=2}.
  \]  
\end{lemma}
\begin{proof}
  It is clear that
  $\norm{(a-1)(b-1)g} = \norm{(a-1)(b-1)}=2$
  when $(a,b)\in \Omega$.
  Conversely, suppose $v\in \latt(A)$ has length~$2$.
  Then $v = x+y-r-s$ with distinct
  $x$, $y$, $r$, $s\in A$ such that $xy=rs$.
  The equation $v=(a-1)(b-1)g$ yields
  \[ g\in \{x,y\} \quad \text{and}\quad
     \{ag,bg\} = \{r,s\}.
  \]
  There are obviously four triples $(a,b,g)$ that fulfill these 
  conditions, and in view of $xy=rs$, every such triple
  solves $v=(a-1)(b-1)g$. 
  As $\norm{v}=2$, we must have $(a,b)\in \Omega(A)$. 
\end{proof}

\begin{cor}\label{c:LAkissnumb}
  Let $A$ be an abelian group
  and let $t$ be the order of the subgroup
  $\menge{a\in A}{a^2 = 1}\leq A$.
  Then 
  \[ \abs[\big]{\menge{v\in\latt(A)}{\norm{v}=2}}
     = \frac{1}{4}\abs{A} \big[ (\abs{A}-1)(\abs{A}-3) + t-1 \big].
  \]
\end{cor}
\begin{proof}
  The size of $\Omega(A)$ is
  $\abs{\Omega(A)}= (\abs{A}-1)(\abs{A}-3) + t-1 $.
\end{proof}

Together with \cref{l:normprodd2a},
this yields:
\begin{cor}
  The minimum distance of $\latt(A)$ is
  $\sqrt{8}$ for $\abs{A}=2$,
  is $\sqrt{6}$ for $\abs{A}=3$
  and is equal to $\sqrt{4}=2$ 
  for $\abs{A}\geq 4$.
\end{cor}
(This result is also contained in the paper 
by Böttcher et al.~\cite{BoettcherFGM15}
as part of their Theorem~1.1.)

\begin{remark}
   The equality $\latt(A) =  (\Delta A)^2$
   suggests two natural generalizations,
   which can of course be combined.
   First, one may study the lattice theoretic properties
   of $(\Delta G)^2$ for a nonabelian group.
   Notice, however, that for a nonabelian group,
   the minimal distance of $(\Delta G)^2$ is
   $\sqrt{2}$: 
   Namely, if $gh\neq hg$, then
   \[ (g-1)(h-1)-(h-1)(g-1) = gh -hg \in (\Delta G)^2
   \]
   has length $\sqrt{2}$.
   
   Second, one may study higher powers $(\Delta A)^r$.
   When $A=C_n$ is cyclic of order $n$, 
   then $(\Delta C_n)^r$ is isometric to the
   \emph{Craig lattice}~$A_{n-1}^{(r)}$
   \cite[Prop.~5.4.5]{Martinet03PLES}.
   When $n=p$ is prime, then it is known that
   $\mindis (\Delta C_p)^r 
    \geq \sqrt{2r}$ for $r < (p-1)/2$,
   and equality holds when $r$ is a proper divisor of $p-1$
   \cite[Theorem~5.4.8]{Martinet03PLES}.
   Examples are known where the inequality is strict.
   When $A$ is not cyclic of prime order, 
   then in general we may have
   $\mindis (\Delta A)^r < \sqrt{2r}$.
   For example, $\mindis(\Delta C_6)^3= \sqrt{4}=2$.
\end{remark}

\section{Basis of minimal vectors}
\label{sec:basmin}

One of the main results of Böttcher et al.~\cite{BoettcherFGM15}
is that $\latt(A)=(\Delta A)^2$ has a lattice basis of elements of 
minimal length, except when 
$A$ is cyclic of order~$4$.
We give an alternative proof of this result in this section.

It is elementary to see that 
$\latt(C_2)$ and $\latt(C_3)$ have a basis 
of minimal vectors,
and that $\latt(C_4)$ has no such basis:
for $C_4$, \cref{c:LAkissnumb} yields
that $\latt(C_4)$ has only four vectors 
of minimal length~$2$, 
and these come in pairs $\pm v$.

We begin our proof with a (well-known) lemma:

\begin{lemma}\label{l:cycbas}
  Let $A= \erz{a}$ be cyclic of order $n$.
  Then the ideal $\Delta A = (a-1)\ints A$ is principal,
  and the following set is a $\ints $-basis of $(\Delta A)^2$:
  \begin{multline*}
   \menge{(a-1)(b-1)}{ 1\neq b\in A} \\
   {}=
   \{\, (a-1)(a-1),\, (a-1)(a^2-1),\,
           \dotsc ,
           (a-1)(a^{n-1}-1) \,\}.
  \end{multline*}
\end{lemma}
\begin{proof}
  We have the identity
  \[ a^k - 1 = (a-1)(a^{k-1} + \dotsb + a + 1).
  \]
  As the elements $a^k-1$ for $k= 1$, $\dotsc$, $n-1$
  form a basis of $\Delta A$, we see that
  $\Delta A = (a-1)\ints A$,
  and that the displayed set forms a basis of 
  $(\Delta A)^2=(a-1)\Delta A$.
\end{proof}

Notice that all but two elements in the basis
    from Lemma~\ref{l:cycbas} have squared
    length~$4$.
    The exceptions are the first 
    and the last element, 
    namely $(a-1)^2$ and $(a-1)(a^{-1}-1)$.
    We will see below how to replace these by shorter vectors
    when $n\geq 5$.

\begin{lemma}\label{l:dirprodbas}
  Let $A=B\times C$ be a direct product of 
  finite, abelian groups.
  Suppose that $D$ is a basis of $(\Delta B)^2$ and 
  $E$ is a basis of $(\Delta C)^2$.
  Define
  \[  M(B,C):= \menge{(b-1)(c-1)}{ b\in B\setminus\{1\},\, 
                           c\in C\setminus\{1\} }.
  \]
  Then $ D\cup E \cup M(B,C)$
  is a basis of $(\Delta A)^2$.
\end{lemma}
\begin{proof}
  The set $ D\cup E \cup M(B,C)$
  has the right cardinality for a lattice basis. 
  Thus it suffices to show that every generator of 
  $(\Delta A)^2 = (\Delta[B\times C])^2$ is in the $\ints$-span
  of $ D\cup E \cup M(B,C)$.
  We know that $(\Delta [B\times C])^2$ is spanned by elements of the
  form
  \[ (b_1c_1-1)(b_2c_2-1), 
     \quad b_1, b_2\in B, c_1, c_2\in C.
  \]
  We have the following identity 
  (which can be verified by direct computation):  
  \begin{align*}
    (b_1c_1-1)(b_2c_2-1) &=
       (b_1b_2-1)(c_1c_2-1) \\
       & \qquad \quad
       {} - (b_1-1)(c_1-1) - (b_2-1)(c_2-1)
       \\
       & \qquad \quad 
       {} + (b_1-1)(b_2-1) + (c_1-1)(c_2-1).
  \end{align*}  
  Here, the first three summands are in
  $M(B,C)$ (up to sign),
  and the last two are in the span of $D$ 
  or $E$, respectively.
  Thus the lemma follows.
\end{proof}

\begin{thm}
   Let $A$ be a finite abelian group.
   \begin{enumthm}
   \item $(\Delta A)^2$ has a lattice basis of elements of the
       form $(a-1)(b-1)$ with $a$, $b\in A$.
   \item \cite[Theorem~1.2]{BoettcherFGM15}
         If either $\abs{A} >4$ or $A\iso C_2\times C_2$, 
         then
         $(\Delta A)^2$ has a lattice basis of elements of
        length~$2$.
   \end{enumthm}       
\end{thm}
\begin{proof}
  Write 
  \[ A = \erz{a_1} \times \dotsm \times \erz{a_r}
  \]
  as a direct product of cyclic groups.
  By Lemma~\ref{l:cycbas},
  $(\Delta\erz{a_i})^2$ has a basis $B_i$ containing elements
  of the form 
  $(a_i-1)(a_i^k-1)$.
  Using Lemma~\ref{l:dirprodbas}, we can extend
  $B_1\cup B_2$ to a basis $B_{1,2}$ of $(\Delta\erz{a_1,a_2})^2$.
  Using Lemma~\ref{l:dirprodbas} again, we can extend 
  $B_{1,2}\cup B_3$ to a basis $B_{1,2,3}$
  of $(\Delta\erz{a_1,a_2,a_3})^2$.
  Repeating, we arrive at a basis of $(\Delta A)^2$
  containing only elements of the form $(a-1)(b-1)$.
  This shows the first part.
  
  To see the second part of the theorem, observe first that
  the set $M(B,C)$ defined in Lemma~\ref{l:dirprodbas}
  contains only elements of squared length $4$.
  However, each $B_i$ contains elements
  of larger squared length, namely
  $(a_i-1)^2$ and $(a_i-1)(a_i^{-1}-1)$
  (possibly equal).
  We need to replace these by smaller elements.
  
  For the moment, write $a=a_i$.
  Choose an element $b\in A$ such that
  $b\notin\{1, a^{\pm 1}, a^2\}$.
  This is possible since we assume that 
  $\abs{A} > 4$ or $A \iso C_2 \times C_2$.
  We have
  \[ (a-1)^2 = (a-1)(a-b) + (a-1)(b-1).
  \]
  The assumption on $b$ assures that both
  $(a-1)(a-b)$ and $(a-1)(b-1)$ have squared length~$4$.
  Moreover, we can choose  $b$ such that $(a-1)(b-1)$ is
  already part of our basis constructed in the first step:
  Namely, when $\abs{\erz{a}} \geq 5$, we can take
  $b=a^3$.
  Otherwise, $A$ is not cyclic, and we can choose 
  $b$ in either $\erz{a_{i+1}}$ or $\erz{a_{i-1}}$.
  (Recall that $a=a_i$.)
  Thus we can replace $(a-1)^2$ by the element
  $(a-1)(b-a)$ and still have a basis.
  
  Similarly, we have
  \[  (a-1)(a^{-1}-1) 
     = (a^{-1}-1)(ab-1) + (a-1)(b-1).
  \] 
  Here we choose $b$ such that 
  $b\notin\{1,a^{\pm 1}, a^{-2}\}$
  and $(a-1)(b-1)$ is part of our basis already constructed.
  Then we can replace $(a-1)(a^{-1}-1)$ by
  $(a^{-1}-1)(ab-1)$.
  This finishes the proof.
\end{proof}

\begin{remark}
    We can even fulfill both conditions of the theorem
    simultaneously:
    of course,
    the element $(a-1)(a-b)$ in the above proof is not of the form
    $(x-1)(y-1)$ with $x$, $y\in A$.
    But in view of the equations
    \begin{align*} 
     (a-1)^2 &= (a-1)(a^{-1}-1) - (a^{-1}-1)(a^2-1) \\
        &= (a-1)(b-1) - (a^2-1)(b-1) + (a-1)(ab-1),
    \end{align*}
    we can also replace $(a-1)^2$ by 
    $(a^{-1}-1)(a^2-1)$ when $\abs{\erz{a}}\geq 4$,
    or by an element of the form $(a-1)(ab-1)$ when
    $A$ is not cyclic.
    
    In particular, for $A=\erz{a}$ cyclic of order
    $n \geq 5$, we have the following basis
    of minimal elements:
    \begin{multline*}
       \{\,(a-1)(a^2-1),\, (a-1)(a^3-1),\, \dotsc, 
           (a-1)(a^{n-2}-1) \,
       \}
       \\
       \cup \{\, (a^{-1}-1)(a^2-1),\, (a^{-1}-1)(a^3-1) \, \}.
    \end{multline*}
    This basis of $\latt(A)=(\Delta A)^2$ 
    (up to an automorphism) 
    was given by M.~Sha~\cite[Theorem~3.2]{Sha15}, 
    with an entirely different proof.
\end{remark}
\begin{remark}
  Another way of constructing a basis of minimal vectors
  goes as follows:
  Suppose that $A= \erz{a} = \erz{b}$. 
  Then by \cref{l:cycbas},
  $(\Delta A)^2 = (a-1)(b-1)\ints A$, 
  and thus 
  \[  \menge{(a-1)(b-1) a^k}{k=0, 1, \dotsc, n-2}
  \]
  is a basis of $(\Delta A)^2$, 
  where $n=\abs{A}$ as before.
  When $\phi(n) > 2$
  (that is, $n  \notin \{ 1, 2, 3, 4, 6\}$),
  then we can choose $b\neq a^{\pm 1}$,
  so that all elements in this basis have squared length~$4$.
  (When $n$ is odd, we can choose $b=a^2$,
  and get the basis given by
  Martinet~\cite[Prop.~5.3.5]{Martinet03PLES}.)
  
  The basis of this remark is contained in a single $A$-orbit,
  while the basis of the previous remark
  contains elements from $n-3$ different $A$-orbits. 
\end{remark}


\section{Eutaxy}
\label{sec:eutaxy}

\begin{defi}\label{def:eutacticset}
    Let $(V, \eskp)$ be a Euclidean space.
    For $s\in V$, let 
    $\pi_s\colon V\to V$ be the linear map defined by
    \[  v\pi_s = \skp{v}{s}s.
    \]
    A finite subset $S\subset V$ is called
    \defemph{eutactic} (in $V$), 
    when there are
    positive scalars $r_s >0$ such that
    \[ \id_V = \sum_{s\in S} r_s \pi_s.    
    \]
    Equivalently, for all $v\in V$ we have
    \[ v = \sum_{s\in S} r_s \skp{v}{s}s.
    \]
    The subset $S$ is called \defemph{strongly eutactic},
    when all $r_s$ can be chosen to be equal.
\end{defi}

When $S\subset V$ is eutactic in $V$, then $S$ spans $V$.
So it is often more interesting to know 
whether a set~$S$ is eutactic in its $\reals$-linear span
$\LinH_{\reals}(S)$.
We have the following simple observation, which follows
from $\Ker(\pi_s)= \erz{s}^{\perp}$:
\begin{lemma}
\label{l:eutaxy_subspace}
    A subset $S\subset V$ is eutactic in  its span $U$
    if and only if there are $r_s> 0$ such that
    \[ \sum_{s\in S} r_s \pi_s = p_U,
    \]
    where $p_U\colon V\to V$ is the orthogonal projection
    from $V$ onto $U$, and $\pi_s\colon V\to V$ is as in
    \cref{def:eutacticset}.
\end{lemma}

\begin{defi}
    A lattice $\Lambda\subset V$ is called 
    (strongly) eutactic  in $V$   
    if the set
    \[ \Mvecs(\Lambda):= \menge{s\in\Lambda}{\norm{s} = \mindis(\Lambda)}
    \]
    is (strongly) eutactic in $V$.  
\end{defi}

Let $G$ be a finite group.
We recall some facts on the group algebra $\reals G$
as a Euclidean space, which we will need.
The group algebra $\reals G$ has an involution, given by
\[ \left( \sum_{g\in G} r_g g\right)^{*}
    = \sum_{g\in G} r_g g^{-1}.
\]
The canonical inner product on $\reals G$ can be defined as
$\skp{r}{s} = \eps(rs^{*})= \eps(s^*r)$, where
$\eps\colon \reals G \to \reals$ sends $\sum_{g\in G}r_g g$ to $r_1$.
Let us write
\[ e_G := \frac{1}{\abs{G}} \sum_{g\in G} g \in \reals G. 
\]
The next result is easy to derive from standard results
in representation theory
(cf. \cite[\S~2.6, Thm.~8]{SerreLRFG}),
but for the sake of completeness, we include a short proof here.
\begin{lemma}\label{l:proj_aug}
    The map $ r \mapsto (1-e_G)r$
    is the orthogonal projection from $\reals G$
    onto 
    \[ \LinH_{\reals} (\Delta G)
      = \menge[\bigg]{\sum_{g\in G} r_g g \in \reals G
                  }{\sum_g r_g = 0}.
    \]
\end{lemma}
\begin{proof}
    Write $\iota\colon \reals G \to \reals$ for the 
    augmentation map (coefficient sum),
    $\iota\left( \sum_{ g\in G } r_g g\right) = \sum_{ g\in G } r_g$,
    so that $\Ker \iota = \LinH_{\reals}( \Delta G)$.
    Clearly $ge_G = e_G = e_G g$ for all $g\in G$. 
    Thus $e_G^2 = e_G$
    and 
    $ e_Gr = \iota(r) e_G $
    for all $r\in \reals G$.
    We can write $r= e_G r + (1-e_G)r$
    for any $r\in \reals G$.
    We have
    $\skp{ e_G r }{ (1-e_G)r } = 
       \eps( r^*(1-e_G) e_G r ) = \eps(0)=0$,
    and $(1-e_G)r$ is in $\Ker \iota = \LinH_{\reals}( \Delta G)$.
    The result follows.
\end{proof}

\begin{lemma}\label{l:orbitsumproj}
  Let $s\in \reals G$ and let $\pi_s\colon \reals G\to \reals G$
  be defined as in \cref{def:eutacticset},
  with respect to the canonical inner product on $\reals G$.
  Then 
  \[ r \sum_{g\in G} \pi_{sg} = ss^{*}r
  \]
  for all $r\in \reals G$,
  that is,
  $\sum_g \pi_{sg}$ is left multiplication with
  $ss^*$.
\end{lemma}
\begin{proof}
    Write $r=\sum_{x\in G} r_x x$ and
    $s=\sum_{y\in G} s_y y$.
    Then
    \begin{align*}
        r\sum_{g\in G} \pi_{sg} 
           &= \sum_{g\in G} \skp{r}{sg} sg
        \\ &= \sum_{y,g\in G} r_{yg}s_y \, sg
        = s \sum_{x,y\in G} r_x s_y \, y^{-1}x
        = s s^* r.
        \qedhere
    \end{align*}
\end{proof}

When $A$ is a finite abelian group and
$A\not\iso C_2$ or $C_3$,
then
\[ \Mvecs(\latt(A)) = S(A) := \menge{v\in \latt(A)}{\norm{v} = 2}.
\]
To show that $\latt(A)$ is eutactic in its span for $A\not\iso C_4$,
we have to show that the set $S(A)$
is eutactic in its span, when
$A \not\iso C_2$, $C_3$ or $C_4$.
(The lattices $\latt(C_2)$ and $\latt(C_3)$ are also eutactic,
but $\Mvecs(\latt(A)) \neq S(A)$ for these groups,
and $\latt(C_4)$ is not eutactic, since
$S(A)=\Mvecs(\latt(C_4))$ spans only a subspace of the span
of $\latt(C_4)$.)

The group $A$ acts by right multiplication on $\reals A$,
and this action preserves the scalar product.
The set $S(A)$ is invariant and decomposes into $A$-orbits.
In \cref{l:minLA_param}, we have described a 
parametrization of
$S(A)$, namely
\[ S(A)
       = \menge{(1-a)(1-b)g}{(a,b)\in \Omega(A), g\in A},
\]
where
\[ \Omega(A) = \menge{ (a,b)\in A\times A
                     }{ a\neq 1 \neq b \neq a^{\pm 1} 
                     }.
\]
(For each element $s\in S(A)$, there are actually
exactly four triples $(a,b,g)\in \Omega(A)$
such that $s=(1-a)(1-b)g$.)
By \cref{l:orbitsumproj}, the problem of showing
that $\latt(A)$ is eutactic is reduced to 
a problem in the group algebra:

\begin{lemma}\label{l:gacrit_eut}
    For $a$, $b\in A$, write $m(a,b) = (a-1)(b-1)$.
    Suppose there are
    $\gamma_{a,b}>0$ for $(a,b)\in \Omega(A)$ such that
    \[ 1-e_A
         = \sum_{(a,b)\in\Omega(A)} 
                \gamma_{a,b} \, m(a,b)m(a,b)^*.
    \]
    Then $S(A)$ is eutactic.
    Moreover, 
    $S(A)$ is strongly eutactic if and only if there is a $\gamma>0$
    such that
    \[ 1-e_A = \gamma \sum_{(a,b)\in\Omega(A)} m(a,b)m(a,b)^* .
    \]
\end{lemma}
\begin{proof}
    If the condition holds,
    then
    \begin{align*}
       r
       \sum_{(a,b)\in\Omega(A)}
       \sum_{g\in G} \gamma_{a,b} \pi_{m(a,b)g}
       &= \sum_{(a,b)\in\Omega(A)} \gamma_{a,b}\, m(a,b)m(a,b)^* r
       \\ &= (1-e_A)r.
    \end{align*}
    On the other hand, the double sum on the left 
    can be rewritten as
    $\sum_{s\in S(A)} \lambda_s \pi_s$, 
    with
    \[ \lambda_s 
         := \sum_{\substack{(a,b)\in\Omega(A)\colon \\
                           \mathclap{\exists g\in A\colon s = m(a,b)g}
                          }
                }
             \gamma_{a,b} 
        \; > 0.
    \]
    By \cref{l:eutaxy_subspace,l:proj_aug}, $S(A)$ is eutactic.
    If the $\gamma_{a,b}$'s are all equal, 
    then the $\lambda_s$'s are all equal and
    $S(A)$ is strongly eutactic.
    Conversely, when $S(A)$ is strongly eutactic,
    then there is $\lambda>0$ such that
    \[ r \sum_{s\in S(A)} \pi_s = \lambda (1-e_A) r
    \]
    for all $r\in \reals G$.
    It follows from \cref{l:minLA_param,l:orbitsumproj} that
    \[ r \sum_{s\in S(A)} \pi_s = 
       (1/4) \sum_{(a,b)\in\Omega(A)} m(a,b)m(a,b)^* r.
    \]
    Thus the result.
\end{proof}

\begin{lemma}
\label{l:sum_mab_om}
    Let $A$ be a finite abelian group, and let
    $T:= \menge{a\in A}{a^2 = 1}$ and
    $S:= \menge{a^2}{a\in A}$.
    Write $m(a,b)=(1-a)(1-b)$ as before.
    Then
    \begin{align*}
      \MoveEqLeft[4]
      \sum_{(a,b)\in\Omega(A)} 
               m(a,b)m(a,b)^*
        \\ 
        &= 4 \abs{A}(\abs{A}-4)(1-e_A) + 4 \abs{A}(1-e_S)
         + 8 \abs{T}(1-e_T).
    \end{align*}
\end{lemma}
\begin{proof}
   Notice that $m(1,b)=m(a,1)=0$ and
   $m(a,a)m(a,a)^* = (1-a)^2(1-a^{-1})^2 = m(a,a^{-1})m(a,a^{-1})^*$.
   This yields 
   \begin{align*}
      \MoveEqLeft[4] \sum_{(a,b)\in\Omega(A)} 
                     m(a,b)m(a,b)^*
       \\
       & = \begin{aligned}[t]
              \sum_{a,b\in A} m(a,b)m(a,b)^*
               &  - \sum_{a\in A} m(a,a)m(a,a)^*
            \\ & - \sum_{a\in A} m(a,a^{-1})m(a,a^{-1})^*
            \\ & + \sum_{t\in T} m(t,t)m(t,t)^* 
           \end{aligned}
       \\ & = \left( \sum_{a\in A} (1-a)(1-a^{-1}) \right)^2
             - 2 \sum_{a\in A} (1-a)^2(1-a^{-1})^2
             + \sum_{t\in T} (1-t)^4.
   \end{align*}
   We have $(1-a)(1-a^{-1}) = 2-a-a^{-1}$ and thus the 
   first sum on the last line equals
   \[ \sum_{a\in A}(2-a-a^{-1}) = 2\abs{A}(1-e_A).
   \]
   Next, we have
   \begin{align*}
   \sum_{a\in A} (1-a)^2(1-a^{-1})^2
      & = \sum_{a\in A} (2-a-a^{-1})^2
      \\ & = \sum_{a\in A} (4 + a^2 + a^{-2} -4a-4a^{-1} + 2)
      \\ & = 6\abs{A} - 8\abs{A} e_A + 2\abs{A} e_S
      \\ & = 8\abs{A}(1-e_A)- 2\abs{A}(1-e_S).      
   \end{align*}
   Finally, for $t^2 = 1$ we have
   $(1-t)^4 = 1 -4t+6-4t+1 = 8(1-t)$ and thus
   \[ \sum_{t\in T} (1-t)^4 = 8 \abs{T}(1-e_T).
   \]
   Therefore,
   \begin{align*}
     \MoveEqLeft \sum_{(a,b)\in\Omega(A)} 
                          m(a,b)m(a,b)^*
          \\
          & = 4\abs{A}^2(1-e_A)^2
           -16\abs{A}(1-e_A) + 4\abs{A}(1-e_S)
           + 8 \abs{T}(1-e_T)
       \\ & = 4 \abs{A}(\abs{A}-4)(1-e_A) + 4 \abs{A}(1-e_S)
                + 8 \abs{T}(1-e_T).\qedhere
   \end{align*}
\end{proof}

As a corollary, we get an alternative proof
of the following result
which was first proved by Böttcher et al.~\cite{BoettcherEFGM19}:
\begin{cor}
\label{c:str_eutac}
   Let $A$ be a finite abelian group of size at least~$4$.
   Then $S(A)$ is strongly eutactic in its span 
   if and only if
   $A$ has odd order or $A$ is elementary abelian.
\end{cor}
\begin{proof}
   By \cref{l:gacrit_eut},
   $S(A)$ is strongly eutactic if and only if
   \[ \sum_{(a,b)\in\Omega(A)} m(a,b)m(a,b)^*
   \] 
   is a positive scalar multiple of $(1-e_A)$.
   Now when $A$ is odd, then $S=A$ and $T=\{1\}$,
   so \cref{l:sum_mab_om} yields
   \[ \sum_{(a,b)\in \Omega(A)} m(a,b)m(a,b)^*
       =\abs{A}(\abs{A}-3)(1-e_A) \, .
   \]
   When $A$ is elementary $2$-abelian, then
   $S=\{1\}$ and $T=A$, and we get 
   \[ \sum_{ (a,b)\in \Omega(A) } m(a,b)m(a,b)^*
      = \abs{A}(\abs{A}-2)(1-e_A) \,.
   \]
   But in all other cases, we have
   $\{1\} < S, T < A$, and the 
   right hand side in \cref{l:sum_mab_om}
   is not a scalar multiple of $(1-e_A)$.
\end{proof}

To show that $\latt(A)$ is eutactic in general,
the strategy is to express the annoying term
$4\abs{A}(1-e_S) + 8\abs{T}(1-e_T)$
in \cref{l:sum_mab_om}
as $\reals$-linear combinations of 
some of the elements $m(a,b)m(a,b)^*$,
with coefficients strictly smaller than~$1$.
By subtracting, we get an expression
of $4\abs{A}(\abs{A}-4)(1-e_A)$ as a positive combination
of the elements $m(a,b)m(a,b)^*$,
as required by \cref{l:gacrit_eut}.

\begin{lemma}
\label{l:subgr_eB_mab}
  Let $B < A$ be a subgroup and $g\in A\setminus B$.
  Write $m(a,b)= (1-a)(1-b)$.
  Then
  \begin{align*}
  \sum_{a,b\in B} m(a,bg)m(a,bg)^*
   &= 4\abs{B}^2(1-e_B)
  \end{align*}
\end{lemma}
\begin{proof}
   We have
   \begin{align*}
      \sum_{a,b\in B} m(a,bg)m(a,bg)^* 
      &= \sum_{a,b\in B}
          (1-a)(1-bg)(1-g^{-1}b^{-1})(1-a^{-1})
      \\ &= \sum_{a\in B} (2-a-a^{-1}) 
            \sum_{b\in B} (2-bg-g^{-1}b^{-1})
      \\ &= 2\abs{B}(1-e_B) \abs{B}(2-e_Bg - g^{-1}e_B)
      \\ &= 2\abs{B}^2 (2-e_B g - g^{-1}e_B -2e_B + e_B g + g^{-1}e_B)
      \\ &= 4\abs{B}^2 (1-e_B). 
      \qedhere
   \end{align*}
\end{proof}

\begin{thm}
\label{t:la_eutactic}
   For each abelian group $A$ not isomorphic to
   $C_4$, 
   the lattice $\latt(A)$ is eutactic in its 
   span $\LinH_{\reals}(\latt(A))$.
\end{thm}
\begin{proof}
   The cases where $\abs{A}\leq 4$ have already been dealt with,
   so we may assume that $\abs{A}> 4$,
   and $\Mvecs(\latt(A)) = S(A)$.
   By \cref{c:str_eutac}, we can assume that $A$ is neither odd nor
   elementary $2$-abelian, and thus
   both $S$ and $T$ from \cref{l:sum_mab_om} are proper subgroups of
   $A$.
   
   Let $B< A $ be a proper subgroup and suppose
   $1\neq b\in B$ and $a\in A\setminus B$.
   Then $(a,b)$, $(b,a)\in \Omega(A)$,
   with $\Omega(A)$ as above and defined in \cref{l:minLA_param}.
   It follows from \cref{l:subgr_eB_mab} that
   \begin{align*}
       \sum_{a\in A\setminus B, b\in B}
         m(a,b)m(a,b)^* 
        &= \sum_{a\in A\setminus B, b\in B}
         m(b,a)m(b,a)^*
        \\ &= ( \abs{A:B}-1 ) 4 \abs{B}^2 (1-e_B).  
   \end{align*}
   It follows that (with notation as in \cref{l:sum_mab_om})
    we can write
    \begin{align*}
       \MoveEqLeft 4\abs{A}(1-e_S) + 8\abs{T}(1-e_T)
       \\
       &= \frac{4\abs{A}}{8\abs{S}^2 (\abs{A:S}-1)}
          \sum_{(a,b)\in S\times (A\setminus S)
                \cup (A\setminus S)\times S
                }
             m(a,b)m(a,b)^*
       \\ 
       & \qquad   + \frac{8\abs{T}}{8\abs{T}^2 (\abs{A:T}-1)}
          \sum_{(a,b)\in T\times (A\setminus T)
                          \cup (A\setminus T)\times T
                          }
                       m(a,b)m(a,b)^*.
    \end{align*}
    Here, an element $m(a,b)m(a,b)^*$ appears with a coefficient
    at most
    \begin{align*}
      \MoveEqLeft[4]
      \frac{4\abs{A}}{8\abs{S}^2 (\abs{A:S}-1)}
      +
      \frac{8\abs{T}}{8\abs{T}^2 (\abs{A:T}-1)}
      \\
      &=
      \frac{\abs{A:S}}{2\abs{S}(\abs{A:S}-1)}
      + \frac{1}{\abs{T}(\abs{A:T}-1)}
      \\
      &\leq \frac{1}{\abs{S}} + \frac{1}{\abs{T}}
       < 1,       
    \end{align*}
    since $\abs{S}\abs{T}=\abs{A}>4$.
    Now it follows from \cref{l:sum_mab_om} that
    we can write
    \[ 4\abs{A}(\abs{A}-4)(1-e_A)
       = \sum_{ (a,b)\in \Omega(A) } 
          \gamma_{a,b} \, m(a,b)m(a,b)^*,
       \quad 0 < \gamma_{a,b} \leq 1.
    \]
    Then, by \cref{l:gacrit_eut},
    we get that $S(A)$ is eutactic,
    which was to be proved.
\end{proof}

R.~Bacher~\cite[\S~5]{Bacher15} has shown that 
$\latt(A)$ is perfect
for $\abs{A}\geq 7$, except when $A\iso C_4\times C_2$.
Together with Voronoi's classical criterion, we get:

\begin{cor}
    Let $A$ be an abelian group with $\abs{A}\geq 7$
    and $A \not\iso C_4 \times C_2$.
    Then $\latt(A)$ is extreme. 
\end{cor}

\printbibliography[heading=bibintoc]
\end{document}